\documentclass[11pt,leqno]{amsart}
\usepackage{amscd}
\usepackage{color}
\usepackage[symbol]{footmisc}
\usepackage{amssymb}
\usepackage{amsfonts}
\usepackage{latexsym}
\usepackage{verbatim}
\usepackage{bbm}
\usepackage[shortlabels]{enumitem}
\usepackage[backref=page]{hyperref}
\usepackage{tikz-cd}
\renewcommand*{\backref}[1]{}
\renewcommand*{\backrefalt}[4]{%
	\ifcase #1 (Not cited).%
	\or        (Cited on page~#2).%
	\else      (Cited on pages~#2).%
	\fi}
\hypersetup{colorlinks=true,linkcolor=blue,citecolor=blue,urlcolor=black}

\newcommand{\Z}{\mathbb{Z}}

\newcommand{\R}{\mathbb{R}}
\newcommand{\C}{\mathbb{C}}
\newcommand{\ad}{\operatorname{ad}}

\newcommand{\g}{\mathfrak{g}}

\newcommand{\msf}[1]{\mathsf{#1}}
\newcommand*{\f}[1]{\mathfrak{#1}}
\renewcommand{\epsilon}{\varepsilon}
\renewcommand{\phi}{\varphi}

\theoremstyle{plain}
\newtheorem{thm}{Theorem}[section]
\newtheorem{prop}[thm]{Proposition}
\newtheorem{lem}[thm]{Lemma}
\newtheorem{cor}[thm]{Corollary}

\theoremstyle{definition}

\newtheorem{rmk}[thm]{Remark}

\title{Long-time existence of the Pluriclosed flow on some fibrations}
\author{Elia Fusi, James Stanfield, and Luigi Vezzoni}

\address[Elia Fusi, Luigi Vezzoni]{Dipartimento di Matematica ``Giuseppe Peano'', Universit\`a di Torino, Via Carlo Alberto 10, 10123 Torino, Italy}
\email{elia.fusi@unito.it}
\email{luigi.vezzoni@unito.it}
\address[James Stanfield]{School of Mathematics and Physics, University of Wollongong,  Northfields Ave,  Wollongong, NSW 2522,  Australia}
\email{jstanfield@uow.edu.au}

\begin{document}

\begin{abstract} We prove long-time existence of the pluriclosed flow on certain compact quotients of Lie groups for non-invariant initial data, as well as on some holomorphic principal torus bundles over  nonpositively curved K\"ahler manifolds.
In particular, our results cover the cases of nilmanifolds and almost-abelian solvmanifolds, and provide a new proof of the long-time existence of the pluriclosed flow on certain complex surfaces, originally established by Garcia-Fernandez, Jordan, and Streets.
These results follow from a general theorem on holomorphic submersions, which is of independent interest and, in particular, also implies the long-time existence of the pluriclosed flow on Oeljeklaus–Toma manifolds, as proved by Streets and Wang.
\end{abstract}

\maketitle

\section{Introduction}
{\em Pluriclosed metrics} were introduced by Bismut in \cite{Bis89} as Hermitian metrics $g$ on  a complex manifold $M$ whose K\"ahler form $\omega$ satisfies 
$$
\partial \bar \partial \omega=0\,. 
$$
  The study of properties of pluriclosed manifolds and the search for examples of such manifolds  were extremely fruitful throughout the last  years.  For instance, in view of \cite{G277},  any compact complex surface admits a pluriclosed metric.  Other examples of pluriclosed metrics were constructed in the locally homogeneous case, see  \cite{BF24,FP22, FPS04, LM25, MS11}, and  on toric fibrations, see \cite{FGV19,GGP08, Swa10}. Further examples can be constructed by blowing-up a compact pluriclosed manifold, see \cite{FinT09}. 
 
 In \cite{ST10},  Streets and Tian proposed the \emph{pluriclosed flow}  as a tool for the study of pluriclosed metrics. Such geometric flow evolves an initial pluriclosed metric $g_0$ according to the following equation: 
$$
\frac{\partial}{\partial t }g_t=-{\rm Ric}^{(2)}_{g_t}+Q^1_{g_t}\,,
$$
where ${\rm Ric}^{(2)}_{g_t}$ is the \emph{second Chern-Ricci tensor} of $g_t$ and $Q^1_{g_t}$ is a  precise quantity  which is quadratic in the  torsion of the Chern connection of $g_t$. It was shown that the fundamental form $\omega_t$ associated to  $g_t$ evolves according to  
$$
\frac{\partial}{\partial t }\omega_t=-\rho_B^{1,1}\,,
$$
where $\rho_B$ is the  \emph{first Bismut-Ricci form} of $g_t$ defined by $\rho_B:=\frac12{\rm tr}(JR^B)$, and $R^B$ is the curvature tensor of the  Bismut connection of $g_t$. As shown in \cite{
ST10},  the pluriclosed flow is always well-posed on compact manifolds  and it belongs to the family of the \emph{Hermitian curvature flows}, see \cite{ST11}. In analogy to what is known for the K\"ahler-Ricci flow, in \cite{ST13}, it was conjectured that the solution of the pluriclosed flow  exists in a maximal interval $[0,\tau(g_0))$, where 
$$
\tau(g_0):=\sup\,\{t>0\,\,|\,\,[\omega_0]_{{\rm A}}-tc_1(M)>0\}\,.
$$
Here,  $[\omega_0]_{{\rm A}}:=\{\omega_0+\partial \bar\alpha+\bar \partial \alpha\,\,|\,\, \alpha \in \Lambda^{1,0}(M)\}$ is the Aeppli cohomology class of $\omega_0$, 
and $c_1(M)$ is regarded as a class in Aeppli cohomology. In particular, according to the above conjecture,  it is  expected that the solution of the pluriclosed flow  on manifolds with $c_1(M)\le 0$ is immortal,   for every initial starting metric. This was confirmed so far in the following cases:
\begin{itemize}
\item manifolds with a Hermitian metric with nonpositive holomorphic bisectional curvature~\cite{Str16};

\vspace{0.1cm}
\item manifolds with certain globally generated bundles \cite{Str216};

\vspace{0.1cm}
\item Bismut-flat manifolds \cite{B23,GFJS23};

\vspace{0.1cm}
\item minimal non-K\"ahler  compact surfaces of Kodaira dimension $\kappa\geq 0$ \cite{GFJS23};

\vspace{0.1cm}
\item Oeljeklaus-Toma solvmanifolds \cite{SW25};
\end{itemize}
while it was confirmed on arbitrary solvmanifolds only for invariant initial metrics \cite{AL19,Bol16,EFV15,FLS24,FV24}. One of the main  purposes of the present paper is to show the long-time existence of the pluriclosed flow on some quotients of Lie groups for non-invariant initial data:

\begin{thm}\label{main3}
Let $M=\msf G / \Gamma$ be the compact quotient of a Lie group by a lattice endowed with a left-invariant complex structure. Assume that $\msf G$ satisfies one of the following assumptions: 
\begin{enumerate}
\item[$1.$] $\msf G$ is nilpotent;

\vspace{0.1cm}
\item[$2.$] $\msf G$ is almost-abelian;

\vspace{0.1cm}
\item[$3.$] $\msf G$ is solvable and the nilradical of its Lie algebra is of codimension $2$, abelian, and invariant under the complex structure.
\end{enumerate}
Then the pluriclosed flow on $M$ has a long-time solution for every initial pluriclosed metric. 
\end{thm}

By definition, a complex structure on a quotient space $M = \msf G / \Gamma$ is \emph{left-invariant} if it lifts to a left-invariant complex structure on $\msf G$. 

We recall that a connected Lie group $\msf G$ is \emph{nilpotent} if it can be covered by a subgroup of unipotent matrices. 
Their compact quotients are called \emph{nilmanifolds}, and a rich subclass of them admit invariant complex structures supporting pluriclosed metrics \cite{AN22, EFV12, FPS04}. An example is the well-known Kodaira--Thurston surface. 

A Lie group is \emph{almost-abelian} if its Lie algebra has a codimension-one abelian ideal. While the algebraic description is simple, many of their compact quotients also provide a large and varied class of pluriclosed manifolds, including  \emph{Endo--Pajitnov} manifolds \cite{COS24, EP20}. Invariant pluriclosed metrics on these groups were characterized in \cite{AL19}.

Lie groups satisfying the third condition in the statement of Theorem \ref{main3} and their invariant pluriclosed metrics were studied in \cite{BF24, CZ24}.


The geometry of nilmanifolds admitting pluriclosed metrics is well understood. In view of \cite{AN22}, the universal cover $\mathsf G$ of $M$ is $2$-step nilpotent and its center is a complex submanifold of $(\msf G,\widetilde J)$, see \cite{EFV12}. Hence, $M$ can be viewed as a holomorphic principal torus bundle over a torus, see \cite{AN22,EFV12}.
Motivated by this, the question of whether the pluriclosed flow could have long-time solutions on total spaces of holomorphic principal  torus bundles over K\"ahler manifolds 
arises naturally. 
 Our  next  main result  answers this question  on some  holomorphic principal  torus bundles over nonpositively curved K\"ahler manifolds.

\begin{thm}\label{surfaces}
Let $M$ be a compact complex manifold which is the total space of a holomorphic principal  $T_{\C}^r$-bundle $\pi\colon M\to X$ over a compact K\"ahler manifold $(X,h)$ with nonpositive  holomorphic bisectional curvature. Assume that there exist connection $1$-forms $\theta_i$ satisfying
$
d\theta_i=\pi^*\omega_i\,,
$
for some $\omega_i\in \Lambda^{1,1}(X)$, such that 
$$
[\omega_i]\in H^{2}(X,\Z)\,,\quad \nabla^h\omega_i=0\,, \quad i=1, \ldots, 2r\,,
$$
where $\nabla^h$ is the Levi-Civita connection of $h$. 
Then the pluriclosed flow on $M$ has a long-time solution for every initial pluriclosed metric. 
\end{thm}

At this stage, the technical assumption on the connection $1$-forms of the bundle is needed to produce a background Hermitian metric with  Chern curvature tensor satisfying  specific conditions, see Theorem \ref{main}. On the other hand, the assumption is easily satisfied when the base is a Riemann surface, and thus we prove the following result, generalizing the result of Garcia-Fernandez--Jordan--Streets  on  compact minimal non-K\"ahler surfaces, see \cite{GFJS23}:

\begin{cor}\label{mario}
Any solution of the  pluriclosed flow on a compact complex manifold which is the  total space of a holomorphic principal torus bundle over a Riemann surface of genus $g\ge 1$ is immortal.  In particular, any solution of the pluriclosed flow  on a compact minimal non-K\"ahler surface of Kodaira dimension $\kappa \geq 0$ exists for all positive times.
\end{cor}
 The long-time existence of the pluriclosed flow on minimal non-K\"ahler surfaces of non-negative Kodaira dimension was conjectured in \cite{Str20} together with a precise conjectural picture on the long-time behaviour of the flow. Partial results on the latter were proven in \cite{Str22} when the initial metric is assumed to be invariant under the action of the fibres.
  Higher dimensional examples to which Theorem \ref{surfaces} can be applied can be found in \cite[Example 4.3]{Swa10}.

\medskip 
We deduce  Theorem \ref{main3} and  Theorem \ref{surfaces} from a general result in the following setting:
\begin{itemize}
\item a compact Hermitian manifold $(M,g)$ with a regular covering $p \colon\widetilde M \to M$; and
\item a Hermitian submersion $\pi\colon(\widetilde M,\widetilde g) \to (X, h)$, where $\widetilde g:=p^*g$, such that any $\gamma\in {\rm Deck}(\widetilde M)$ sends fibres of $\pi$ to fibres of $\pi$. 
\end{itemize}
We further denote by $\mathcal V := \ker \mathrm{d} \pi$ the vertical bundle  and by $\mathcal H := {\mathcal V}^{\perp_{\widetilde g}}$ its orthogonal complement with respect to $\widetilde g$. Finally, we denote by $R$ the Chern curvature tensor, and define the holomorphic bisectional curvature of a metric as the holomorphic bisectional curvature of its Chern connection. With this setup, our next result is as follows.

\begin{thm}\label{main}
Assume that $h$ has nonpositive holomorphic bisectional curvature and  
    \begin{equation}\label{eqn_conditions}
    R^{\widetilde g}( V,H,\cdot,\cdot) = R^{\widetilde g}(\cdot,\cdot, V, H) = 0,\qquad R^{\widetilde g}(W_1,\overline{W_1},W_2,\overline{W_2}) \leq 0,
    \end{equation}
for every $V\in \Gamma(\mathcal V), W_i \in \Gamma(\mathcal V^{1,0})$, $i =1,2$, and $H\in \Gamma(\mathcal H)$. Then the pluriclosed flow on $M$ has a long-time solution for every initial pluriclosed metric. 
\end{thm}

The proof of Theorem \ref{main} is based on an iterated Schwarz Lemma adapted to the submersion structure, see Section \ref{secproof1.1}. This generalizes the arguments of Streets--Wang in \cite{SW25} and Garcia-Fernandez--Jordan--Streets in \cite{GFJS23}.
Let $g_t$ be a solution of the pluriclosed flow on $M$. Then by \eqref{eqn_conditions}, its lift $\widetilde g_t$ to $\widetilde M$ satisfies
$$
R^{\widetilde g}(\widetilde g_t^{-1}, \widetilde g_t^{-1})\leq C({\rm tr}_{\widetilde g_t}\widetilde g)({\rm tr}_{\widetilde{g_t}|_{\mathcal H}}\widetilde g|_{\mathcal H})\,,
$$
which,   in view of \cite{Str16},  implies that 
$$
\left(\frac{\partial}{\partial t }-\Delta_{\widetilde g_t}\right){\rm tr}_{\widetilde g_t} \widetilde g\leq C({\rm tr}_{\widetilde g_t}\widetilde g)({\rm tr}_{\widetilde{g_t}|_{\mathcal H}}\widetilde g|_{\mathcal H})\,.
$$
Then the Schwarz Lemma \cite{Lu68} applied to the map $\pi$  together with the fact that $p$  sends fibres of $\pi$ to fibres of $\pi$ implies that ${\rm tr}_{\tilde{g_t}|_{\mathcal H}}\tilde g|_{\mathcal H}$ is bounded.  Using now that both $\tilde g_t$ and $\tilde g$ are lifts of metrics on $M$,   we can  deduce that ${\rm tr}_{ g_t}  g\leq C{\rm e}^{Ct}$.  The long-time existence  of the solution now follows from the higher regularity of uniformly parabolic solutions of pluriclosed flow \cite{GFJS23,JS20,Str16}.

 \medskip 
 We conclude by describing how Theorem \ref{main3} and Theorem \ref{surfaces} can be directly obtained applying Theorem \ref{main}. 
 In each case of Theorem \ref{main3}, the Lie algebra  $\mathfrak g$ of $\mathsf G$ admits an abelian ideal $\mathfrak{i}$ such that 
$$
[\ad_X, J]|_{\f i }=0\,, \mbox{ for every $X\in \g$}\,.
$$
Theorem \ref{main} is then applied to $\pi\colon \msf G\to \msf G/\msf I $, where $\msf I$ is the Lie subgroup of $\msf G$ tangent to $\f i$. 
In the nilpotent case, $\f i$ is the centre of  $\g$  and any Hermitian inner product on $\g$ induces an invariant Hermitian metric on $\msf G$ satisfying \eqref{eqn_conditions}, see Lemma \ref{leftinvcurv}. In both cases 2 and 3 of Theorem \ref{main3}, $\f i$ has codimension $2$ and a metric $g$ satisfies \eqref{eqn_conditions} whenever  
\begin{equation}\label{cond}
g([X,JX],{\rm ad}_XY+J{\rm ad}_{JX}Y)=0\,,\mbox{ for every }Y\in \f i \mbox{ and }X\in \f i^{\perp}\,. 
\end{equation}
In contrast to the nilpotent case, condition \eqref{cond} is not satisfied by every left-invariant metric, but only for a specific choice of $g$. The existence of such a metric in the almost abelian case is proved in \cite{Par21}.

Finally, we obtain Theorem \ref{surfaces} by applying Theorem \ref{main} to $\widetilde M=M$ equipped with the Hermitian metric 
$$
g=\pi^*h+\sum_{i=1}^{2r}\theta_i\otimes\theta_i\,,
$$
where $\theta_i$  are the connection $1$-forms of the fibration.  With this choice, the condition on the curvature forms of the bundle being parallel with respect to $h$  is  shown to be sufficient for $ g $ to satisfy condition \eqref{eqn_conditions} of Theorem \ref{main}.
\subsection*{Acknowledgements}
The authors would  like to thank Mario Garcia--Fernandez, Ramiro A. Lafuente and Jeffrey Streets  for useful conversations and suggestions.
  Part of the present paper was developed during the program \lq\lq  Analytic and Geometric Methods on Complex Manifolds\rq\rq \,in the mathematical research institute MATRIX in Australia. The authors would like to express their gratitude to the organizers and MATRIX staff  for providing a wonderful and fruitful workplace. 
\section{Proof of Theorem \ref{main}}\label{secproof1.1}

We provide in this section the details of the proof of Theorem \ref{main}. The idea is to use the Schwarz Lemma twice to obtain a time-dependent lower bound on the pluriclosed metrics $g_t$, as was done by Streets--Wang in \cite{SW25} for Oeljeklaus--Toma manifolds, and Garcia-Fernandez--Jordan--Streets for minimal non-K\"ahler surfaces \cite{GFJS23}.

\begin{proof}[Proof of Theorem $\ref{main}$]
Let $\{g_t\}_{t\in [0,T)}$ be a solution to the pluriclosed flow on $M^n$. We show that 
$$
{\rm tr}_{g_t}g\leq C{\rm e}^{Ct}\,,\quad \mbox{ for all }t\in [0,T)\,,
$$
where $C$ depends only on $g$ and $g_0$. In particular, $g_t \geq C e^{-Ct}g$ and from \cite{Str16} it follows that maximal-time solutions to the pluriclosed flow are defined for $t\in [0,+\infty)$. To estimate ${\rm tr}_{g_t}g$, we will use the inequality
\begin{equation}\label{eqn_SchwarzGen}
\left(\frac{\partial}{\partial t }-\Delta_{g_t}\right){\rm tr}_{g_t}g\le R^g(g_t^{-1}, g_t^{-1})\,,
\end{equation}
where $R$ is the Chern curvature of $g$, which follows from \cite[Lemma 6.2]{Str16}.

Now, consider the lift $\widetilde g_t$ of $g_t$ to $\widetilde M$. 
For $x\in \widetilde M$ we have the splitting $T_x\widetilde M=\mathcal H_x\oplus \mathcal{V}_x$ and we can fix a $\widetilde g$-unitary basis $\{Z_1, \ldots, Z_n\}$ of $T^{1,0}_x\widetilde M$ such that $\{Z_1, \ldots, Z_k\}$ is a basis of $\mathcal H_x^{1,0}$  while $\{Z_{k+1}, \ldots, Z_{n}\}$ is a  basis of  $\mathcal V_x^{1,0}$. 
By using  the first two conditions in \eqref{eqn_conditions} at $x$, we have 
\begin{align*}
R^{\widetilde g}(\widetilde g_t^{-1}, \widetilde g_t^{-1})&=\sum_{i,j,p, q=1}^n\widetilde g_t^{\, i\bar j} \widetilde g_t^{\,p\bar q}R^{\widetilde g }_{i\bar j p \bar q}\\
&=\sum_{i,j,p,q=1}^k \widetilde g_t^{\, i\bar j}\widetilde g_t^{\, p\bar q}R^{\widetilde g}_{i\bar j p \bar q} + \sum_{i,j,p,q={k+1}}^n \widetilde g_t^{\, i\bar j}\widetilde g_t^{\, p\bar q}R^{\widetilde g}_{i\bar j p \bar q}\\
&\qquad 
+\sum_{i,j=1}^k \sum_{p,q=k+1}^n \widetilde g_t^{\,i\bar j}\widetilde g_t^{\, p\bar q}\left(R^{\widetilde g}_{ i\bar j p \bar q} + R^{\widetilde g}_{p \bar q  i\bar j }\right)\,.
\end{align*}
Since $\widetilde g$ submerses onto a metric $h$ with non-positive holomorphic bisectional curvature, and curvature increases in quotient bundles (see \cite[Ch.0 \S5]{GH94}), the first term in the last equality is non-positive. Moreover, by the third assumption in \eqref{eqn_conditions}, the second term is likewise non-positive. Thus, we find that

\begin{equation}\label{curvestimat}
R^{\widetilde g}(\widetilde g_t^{-1}, \widetilde g_t^{-1}) \le C({\rm tr}_{\widetilde{g_t}|_{\mathcal V}}\widetilde g |_{\mathcal V})({\rm tr}_{\widetilde{g_t}|_{\mathcal H}}\widetilde g|_{\mathcal H})\le C({\rm tr}_{\widetilde g_t}\widetilde g)({\rm tr}_{\widetilde{g_t}|_{\mathcal H}}\widetilde g|_{\mathcal H})\,,
\end{equation} 
where $C$ is a constant depending only on the background Hermitian metric $\widetilde g$. The next step is to estimate the term ${\rm tr}_{\widetilde{g_t}|_{\mathcal H}}\widetilde g|_{\mathcal H}$. 
To do that, we first note that  ${\rm tr}_{\widetilde{g_t}|_{\mathcal H}}\widetilde g|_{\mathcal H}=|\partial \pi|^2_t$, where $\partial \pi \in \Gamma(\Lambda^{1,0}\widetilde M \otimes \pi^*T^{1,0}X)$ is the $(1,0)$-part of the differential of $\pi \colon \widetilde M \to X$ and the time-dependent Hermitian metric on $\Lambda^{1,0}\widetilde M \otimes \pi^*T^{1,0}X$ is defined by:
\[
\langle A , B \rangle_t := \langle A , B \rangle_{\widetilde g_t^{-1},\pi^* h} = \widetilde g_t^{\, i\bar j}h_{p\bar q }A_i^p\overline{B_j^q}\,, \qquad A,B \in \Gamma(\Lambda^{1,0}\widetilde M \otimes \pi^*T^{1,0}X).
\]
Indeed,  with the identification $\mathcal{H}^{1,0} \simeq \pi^*T^{1,0}X$, the map $\partial \pi$ reads as the $\widetilde g$-orthogonal projection from $T^{1,0}\widetilde M$ onto $\mathcal{H}^{1,0}$ and $|\partial \pi|_t^2=\sum_{i,j=1}^{k} \widetilde g^{\,i\bar j}_{t} \widetilde g_{i\bar j}={\rm tr}_{\widetilde g_t|_{\mathcal H}}\widetilde g|_{ \mathcal{H}}$ as claimed. We now compute the heat operator acting on $|\partial \pi|_t^2$. First, using that the holomorphic bisectional curvature of $h$ is non-positive and by applying the Schwarz Lemma for holomorphic maps between complex manifolds \cite{BS23, Lu68},  we get 
$$
\Delta_{\widetilde g_t}|\partial \pi|_{t}^2
\geq |\hat{\nabla} \partial \pi|_t^2+ \langle\partial \pi\circ {\rm Ric}_{\tilde g_t}^{(2)}, \partial \pi\rangle_t
\ge \langle\partial \pi\circ {\rm Ric}_{\widetilde g_t}^{(2)}, \partial \pi\rangle_t\,, 
$$ where
${\rm Ric}_{\widetilde g_t}^{(2)}$ is the endomorphism associated with second Chern--Ricci form of the metric $\widetilde g_t$. 
Since,  for every $t\in [0,T)$,  the metric $\widetilde g_t$ is pluriclosed, we have ${\rm Ric}_{\widetilde g_t}^{(2)}=({\rm Ric}_{\widetilde g_t}^{B})^{1,1}+Q_{\widetilde g_t}^1$, see  \cite{ST10}, where $Q_{\widetilde g_t}^1\ge 0 $.  Hence $\Delta_{\widetilde g_t}|\partial \pi|_t^2\ge \langle\partial \pi\circ ({\rm Ric}_{\widetilde g_t}^{B})^{1,1}, \partial \pi\rangle_t\,.$ Moreover, 
\begin{align*}
\frac{\partial}{\partial t} |\partial \pi|^2_t &= \frac{\partial}{\partial t}{\rm tr}_{\widetilde g_t|_{\mathcal H}}(\widetilde g_{|\mathcal{H}}) =\sum_{i,j=1}^n\sum_{s,l=1}^k\widetilde g_t^{\,s\bar j }\widetilde g_t^{\, i\bar{l}}\widetilde g_{\, s \bar l}(({\rm Ric}^B_{\widetilde g_t})^{1,1})_{i\bar{j}}\\
&= \sum_{i=1}^n\sum_{s,l=1}^k\widetilde g_t^{\, i\bar l}h_{s \bar l}(({\rm Ric}^B_{\widetilde g_t})^{1,1})_i^{s} = \langle\partial \pi\circ ({\rm Ric}_{\widetilde g_t}^{B})^{1,1}, \partial \pi\rangle_t,
\end{align*}
thus $\left(\frac{\partial}{\partial t} - \Delta_{\widetilde g_t}\right)|\partial \pi|^2_t \leq 0$.

We now show that the time varying function $|\partial \pi|^2_t \colon \widetilde M \to \R$ is invariant under the deck transformation group of $\widetilde M$, hence it descends to a function on $M$. Indeed, since any deck transformation $\gamma \in \operatorname{Deck} (\widetilde M)$ sends fibres to fibres and acts by isometries of $\widetilde g$, we have that ${\rm d}\gamma(\mathcal V) \subseteq \mathcal V$ and ${\rm d}\gamma (\mathcal H) \subseteq \mathcal H$. Since $\gamma$ is also an isometry of $\widetilde g_t$, a simple computation shows that $|\partial \pi|_t^2 \circ \gamma = |\partial \pi|^2_t$. Thus, as the covering $p$ is regular, $|\partial \pi|^2_t = p^* f_t$ for a time-varying function $f_t \colon M \to \R$ which therefore satisfies $\left(\frac{\partial}{\partial t} - \Delta_{g_t}\right)f_t\leq 0$. By compactness of $M$, $f_t$ is uniformly bounded in time by the maximum principle, hence so is $ {\rm tr}_{\widetilde{g_t}|_{\mathcal H}}\widetilde g|_{\mathcal H} =|\partial \pi|^2_t =p^*f_t$. Now \eqref{eqn_SchwarzGen}  and \eqref{curvestimat} imply that 
\[
p^*\left(\frac{\partial}{\partial t} - \Delta_{g_t}\right){\rm tr}_{g_t}g =\left(\frac{\partial}{\partial t} - \Delta_{\widetilde g_t}\right){\rm tr}_{\widetilde g_t}\widetilde g\le  R^{\widetilde g}(\widetilde g_t^{-1},\widetilde g_t^{-1}) \leq C {\rm tr}_{\widetilde g_t}\widetilde g = C p^*{\rm tr}_{g_t}g\,.
\]
Another application of the maximum principle yields the bound ${\rm tr}_{g_t}g \leq C e^{Ct}$, as required.
\end{proof}

\section{Proof of Theorem \ref{main3}}
In this section we assume that $M = \msf G / \Gamma$ is a compact quotient of a simply-connected Lie group $\msf G$ endowed with a left-invariant complex structure $J$  and let $\f g$ be the Lie algebra of $\msf G$. The goal of this section is to prove Theorem \ref{main3}. Each case will follow from the next general proposition.

\begin{prop}\label{main1}
Let $(g,J)$ be a Hermitian structure on $\g$ and  let $\f i$   be  a $J$-invariant ideal of $\f g$ such  that 
\begin{equation}\label{bracket}
[\ad_X, J]|_{\f i }=0\,, \mbox{ for every $X\in \g$\,.}
\end{equation}
Assume that the Chern curvature tensor $R$ of $g$ satisfies
\begin{eqnarray}
&& R(U , \bar V , Z,\bar W)=0\,,\label{1}\\ 
&& R(U, \bar U, W,\bar W)+|[U,\bar W]^{0,1}_{\f i}|^2_{g}\leq 0\label{2}\,,
\end{eqnarray}
for any  $Z\in\mathfrak{i}^{1,0}$, 
$U,V,W\in (\mathfrak{i}^{\perp})^{1,0}$, where the subscript $\f i$ denotes the orthogonal projection onto $\f i\otimes \C$. Then any solution of the  pluriclosed flow on $(M,J)$ exists for all positive times.
\end{prop}

Note that, if we consider $\f i=0$, then condition \eqref{1} is always satisfied, while condition \eqref{2} says that $g$ has non-positive holomorphic bisectional curvature. On the other hand, if we take $\f i=\g$, then condition \eqref{bracket} implies that $g$ is Chern-flat. In both cases Proposition \ref{main1} is obtained by applying  \cite[Theorem 1.1]{Str16}.

In order to prove Proposition \ref{main1}, we consider the following preliminary lemma giving an explicit expression of the Chern curvature of an arbitrary left-invariant metric on $\mathsf G$.  
\begin{lem}
Let $(g,J)$ be a Hermitian structure on a Lie algebra  $\g$. Then the Chern curvature tensor $R$ of $g$ takes the following expression
\begin{equation}\label{R}
R(U, \bar V, Z,\bar W)=g([U,\bar W]^{0,1},[Z,\bar V])+\sum_{l=1}^n
g([U,Z_{\bar l}],Z)g([\bar V,Z_l],\bar W)
+g([[U,\bar V]^{1,0},\bar W],Z)\,,
\end{equation}
for every $U,V,Z,W\in \g^{1,0}$, where $\{Z_1, \ldots, Z_n\}$ is a left-invariant  unitary frame with respect to $g$. In particular, 
\begin{equation}\label{RR}
R(U, \bar U, Z,\bar Z)=-|[U,\bar Z]^{0,1}|_g^2+\sum_{l=1}^n
|g([U,Z_{\bar l}],Z)|^2
+g([[U,\bar U]^{1,0},\bar Z],Z)\,.
\end{equation}
\end{lem}
\begin{proof}
Writing $R(Z_{i},Z_{\bar j})Z_{r}=R_{i\bar{j}r}\,\!^{s }Z_s$, we have 
$$
R_{i\bar{j}r}\,\!^{s }=(\nabla_{i}\nabla_{\bar j}Z_r)^s-(\nabla_{\bar{j}}\nabla_{i}Z_{r})^s-(\nabla_{[Z_i,Z_{\bar j}]}Z_{r})^s=
\sum_{l=1}^n\left(\Gamma_{\bar j r}^l\Gamma_{il}^s-\Gamma_{ir}^l\Gamma_{\bar jl}^s-\mu_{i\bar j}^l\Gamma_{lr}^{s}\right)\,,
$$
where the $\Gamma$'s are the Christoffel symbols of the Chern connection $\nabla$ of $(g,J)$ with respect to the frame $\{Z_1, \ldots, Z_n\}$ and the $\mu$'s denote the components of the brackets.  
In order to write the components of $R$ in terms of $\mu$,  we observe that
$$
\Gamma_{ir}^l=g(\nabla_{i}Z_r,Z_{\bar l})=-g(Z_r,\nabla_{i}Z_{\bar l})=-g(Z_r,[Z_{i},Z_{\bar l}])=\mu_{\bar li}^{\bar r}
$$
and 
$$
\Gamma_{\bar jl}^s=g(\nabla_{\bar j}Z_l,Z_{\bar s})=g([Z_{\bar j},Z_l],Z_{\bar s})=\mu_{\bar jl}^s\,.
$$
It follows 
\begin{equation}
R_{i\bar{j}r}\,\!^{s }=\sum_{l=1}^n\left(\mu_{\bar j r}^l\mu_{\bar si}^{\bar l}-\mu_{\bar li}^{\bar r}\mu_{\bar jl}^s-\mu_{i\bar j}^l\mu_{\bar sl}^{\bar r}\right)
\end{equation}
which can be rewritten as 
\begin{multline*}
R(Z_i,Z_{\bar{j}},Z_r,Z_{\bar s})=g([Z_{\bar s},Z_i]^{0,1},[Z_{\bar j},Z_r])
-g([Z_{\bar l},Z_i],Z_{r})g([Z_{\bar j},Z_l],Z_{\bar s})\\
-g([Z_{\bar s},[Z_i,Z_{\bar j}]^{1,0}],Z_r)\,.
\end{multline*}
The claim now follows. 
\end{proof}

\begin{lem}\label{lemma32}
Let $(g,J)$ be a Hermitian structure on a Lie algebra $\g$ and assume that $\g$ has a $J$-invariant ideal $\mathfrak{i}$. Then condition \eqref{bracket} implies that 
\begin{equation}\label{RV}
R(V,\cdot,\cdot, \cdot)=0\,,  \mbox{ for every $V\in \f i$.}
\end{equation}

If further $(\f g, J)$ admits a pluriclosed metric $h$, then condition \eqref{RV} implies that  $h|_{\f i }$  is a K\"ahler metric on $\f i$.  
\end{lem}
\begin{proof} 
The Chern connection  of $g$ takes the following general expression 
 $$
 \begin{aligned}
2g(\nabla_XY, Z)=&\, g([X, Y], Z)+g([Z, X], Y)+g(J[JX, Y], Z)-g(J[JX, Z], Y)\\
=&\, -g((\ad_Y+J\ad_YJ)X, Z)+ g((\ad_Z+J\ad_ZJ)X, Y)\,,
\end{aligned}
$$
for  $X, Y, Z\in\f g $. Now, if $X, Z \in \f i$ and $Y, W\in \f i^{\perp}$, we have that 
\begin{equation}\label{covariantX}
g(\nabla_XY, Z)=-\frac12g(J[\ad_Y, J]X, Z)\,,  \quad g(\nabla_XY, W)=0\,,
\end{equation} concluding that $\nabla_XY=-\frac12J[\ad_Y, J]X$.  We can then compute
$$
\begin{aligned}
R(X, JX, Y, JY)=&\, g(\nabla_X\nabla_{JX}Y-\nabla_{JX}\nabla_{X}Y-\nabla_{[X, JX]}Y,JY )\\
=&\, 2g(\nabla_XY, J\nabla_{JX}Y)- g(\nabla_{[X, JX]}Y, JY)\,.
\end{aligned}
$$ 
Using the second equation in \eqref{covariantX}, we have that $g(\nabla_{[X, JX]}Y, JY)=0$. Hence, 
for any $X\in\f i$, 
$$
R(X,  JX, Y, JY)=\frac12g([\ad_Y, J]X,J[\ad_Y, J]JX )=\frac12|[\ad_Y, J]X|_g^2=0\,,
$$ 
since $[\ad_Y, J]|_{\f i}=0$, for any $Y\in \f g$.  This fact allows us to infer that, for any $X, Y\in\f i$,  $\nabla_XY\in\f i$ and $\nabla_XY=\bar \nabla_XY$, where $\bar \nabla$ is the Chern connection of  $(\f i, J|_{\f i}, g|_{\f i})$. Hence, we have that 
$$
R(X, JX, Y, JY)=\bar R(X, JX, Y, JY)\,,
 $$ where $\bar R$ is the curvature tensor of $\bar \nabla$. Thus, using that $[\ad_X, J]|_{\f i}=0$, for any $X\in \f i$, we can deduce that   $R(X, JX, Y, JY)=0$, for any $X, Y \in\f i$, see \cite{B58}, 
 giving the first part of the statement. 
 
If $(\f g, J)$ has pluriclosed metric $h$, then its restriction to $\f i$ is pluriclosed with respect to $J|_{\f i}$, and, if we assume \eqref{RV}, $h|_{\f i}$ is Chern-flat. Since any pluriclosed Chern-flat metric is K\"ahler, the claim follows. 
\end{proof}

\begin{rmk}
Note that Lemma \ref{lemma32} asserts that condition \eqref{RV} is recovered from a condition which is independent of the choice of the metric.  
\end{rmk}

\begin{proof}[Proof of Proposition $\ref{main1}$]
The claim is implied by Theorem \ref{main}. Indeed, $\msf G\to  M$ is a regular covering of $M$ and the Hermitian structure $(g,J)$ on $M$ lifts to an invariant Hermitian structure on $\msf G$, which  will   still be denoted by $(g,J)$. Then, we  can consider the subgroup $\msf I$ of $\msf G$ whose Lie algebra is $\mathfrak i$. Since $\mathfrak i $ is an ideal, $\msf I$ is a normal subgroup of $\msf G$ and the quotient $X:=\msf G/\msf I$ is a Lie group whose Lie algebra can be identified, as a vector space,  with $\f i^{\perp}$. Furthermore $X$ inherits the Hermitian structure given by the restriction of $J$ to $\f i^{\perp}$ and the metric $h:=g|_{\f i^{\perp}}$. By construction the canonical  
projection onto the quotient
$$
\pi\colon \msf G\to X
$$
is a Hermitian submersion. Next,  we show that assumptions \eqref{eqn_conditions} in Theorem \ref{main} are satisfied. In view of Lemma \ref{lemma32},  condition \eqref{bracket} forces $ R(V,\cdot,\cdot,\cdot)=0$,  for every $V\in \mathfrak i$ (here $\f i$ plays the role of the vertical bundle). Now, fix a unitary frame $\{Z_1,\ldots,  Z_n\}$ for $\f g$ with respect to $g$ such that its first $k$ vectors span $\f i^{1,0}$. Using that $\pi$ is a Hermitian submersion, $\{\pi(Z_{k+1}), \ldots, \pi(Z_n)\}$ is a unitary frame for $\f g/\f i $ with respect to $h$. Then, for any $U,W\in \f (\f i^{\perp})^{1,0}$,   by using \eqref{RR} with $\pi(U)$ and $\pi(W)$ and the fact that $\pi $ is both a Hermitian submersion and a Lie algebra homomorphism, we have 
$$
(\pi^*R^h)(U, \bar U, W,\bar W)=-|[U,\bar W]^{0,1}_{\f i^{\perp}}|_g^2+\sum_{l=k+1}^n
|g([U,Z_{\bar l}],W)|^2
+g([[U,\bar U]^{1,0}_{\f i^{\perp}},\bar W],W)\,.
$$
$$
(\pi^*R^h)(U, \bar U, W,\bar W)=R(U, \bar U, W,\bar W)+|[U,\bar W]^{0,1}_{\f i}|_g^2-\sum_{l=1}^k
|g([U,Z_{\bar l}],W)|^2
+g([[U,\bar U]^{1,0}_{\f i},\bar W],W)\,.
$$
Since $\f i$ is an ideal of $\g$, we have $g([U,Z_{\bar l}],W)=0,$ for every $l=1,\dots,k$. Finally,  condition \eqref{bracket} implies 
$$
[Y,Z]\in \f i^{1,0}\,,\quad \mbox{ for all }Z\in \f i^{1,0}, Y \in \mathfrak{i}\otimes \C\,.
$$
Then,  
$g([[U,\bar U]^{1,0}_{\f i},\bar W],W)=0$ and 
$$
(\pi^*R^h)(U, \bar U, W,\bar W)=R(U, \bar U, W,\bar W)+|[U,\bar W]^{0,1}_{\f i}|_g^2\,.
$$
Hence, condition \eqref{2} implies that $h$ has nonpositive holomorphic bisectional curvature and the claim follows from Theorem \ref{main}. 
\end{proof}

We prove Theorem \ref{main3} by studying separately the various cases. 

\subsection{Nilmanifolds}
 In this subsection, we specialize our analysis in the case in which $\msf G$ is nilpotent.  In this case, we can prove the following.  

\begin{lem}\label{leftinvcurv}
Let $(\g,J)$ be a $2$-step nilpotent Lie algebra equipped with a complex structure. Let $\f i$ be the centre of $\g$ and assume that $J\f i = \f i$. Then 
\begin{equation*}
[\ad_X, J]|_{\f i }=0\,,  \mbox{ for every }X\in \g\,.
\end{equation*}
Moreover, if $g$ is a Hermitian inner product on $\g$,  its Chern curvature satisfies 
\begin{eqnarray}
&& R(\cdot , \cdot , Z,\bar W)=0\,, \label{11}\\
&& R(U, \bar U, W,  \bar W)=-|[U,\bar W]^{0,1}|^2_g\,, \label{22}\\
&& R(U, \bar U, Z,  \bar Z)=\sum_{l=1}^n\,|g([U,Z_{\bar l}],Z)|^2\ge 0\,,\label{33}
\end{eqnarray}
for any $U, V\in \g^{1,0}$, $Z\in\f i^{1,0}$, $W\in (\f i^{\perp})^{1,0}$\,, where $\{Z_1, \ldots, Z_n\}$ is a unitary basis of $\g$ with respect to $g$. In particular, 
\begin{equation}\label{flat}
R(U, \bar U, W,\bar W)+|[U,\bar W]^{0,1}|^2_{g}=0\,,
\end{equation}
for any $U,W\in (\f i^{\perp})^{1,0}$. 
\end{lem}

\begin{proof}
The condition $[\ad_X, J]|_{\f i }=0$ is trivial since we are assuming the centre $\f i$ is $J$-invariant. Moreover, using that $\f g$ is $2$-step nilpotent and $\f i$ is $J$-invariant,   for $U,V,T,S\in \g^{1,0}$, $Z\in \f i^{1,0}$ and $W\in (\f i^{\perp})^{1,0}$, we have  
\begin{eqnarray*}g([U,\bar W]^{0,1},[Z,\bar V]^{1,0})=0\,, \quad 
g([U,\bar T],Z)g([\bar V,S],\bar W)=0\,,\quad 
g([[U,\bar V]^{1,0},\bar W],Z)=0\,.
\end{eqnarray*}
 Hence,  \eqref{11} follows from  \eqref{R}.
 Furthermore, for $U \in \g^{1,0}$, $Z\in\f i^{1,0}$, $W\in (\f i^{\perp})^{1,0}$ we have 
$$
\begin{aligned}
R(U, \bar U, W,  \bar W)=&\,g([\bar W,U],[\bar U,W]^{1,0})+\sum_{l=1}^n
g([U,Z_{\bar l}],W)g([\bar U,Z_l],\bar W)
-g([\bar W,[U,\bar U]^{1,0}],W)\\
=&\,-|[U,\bar W]^{0,1}|^2_g
\end{aligned}
$$
and 
$$
\begin{aligned}
R(U, \bar U, Z,  \bar Z)=&\,g([\bar Z,U],[\bar U,Z]^{1,0})+\sum_{l=1}^n
g([U,Z_{\bar l}],Z)g([\bar U,Z_l],\bar Z)
-g([\bar Z,[U,\bar U]^{1,0}],Z)\\
=&\,\sum_{l=1}^n\,|g([U,Z_{\bar l}],Z)|^2\,,
\end{aligned}
$$
 from which the claim follows. 
\end{proof}

\begin{proof}[Proof of Theorem $\ref{main3}-1$]
Let $M= \msf G/ \Gamma$ be a nilmanifold endowed with a left-invariant complex structure and admitting a pluriclosed metric. Then $M$ has a left-invariant pluriclosed metric \cite{FG04}, the Lie algebra $\f g$ of $\msf G$ is $2$-step nilpotent \cite{AN22} and its centre $\f i$ is $J$-invariant \cite{EFV12}. Lemma \ref{leftinvcurv} implies that all the assumptions of Proposition \ref{main1} are satisfied, concluding  the proof.  
\end{proof}

\subsection{The case ${\rm codim}_\R\f i=2$} In this subsection we prove  cases $2$ and $3$ of Theorem \ref{main3}. In both cases the ideal $\f i$ has real codimension $2$ and condition \eqref{bracket} is satisfied. We consider the following general lemma.

\begin{lem}\label{codim2}
Let  $(\g,J)$ be a Lie algebra with a complex structure and let $\f i$ be a $J$-invariant ideal of $\g$ such that ${\rm codim}_\R\f i=2$ and 
\begin{equation*}
[\ad_X, J]|_{\f i }=0\,,\quad \mbox{ for all $X\in \g$}\,.
\end{equation*}
Then the Chern curvature of a Hermitian metric $g$ on $(\g,J)$ satisfies 
\begin{eqnarray}
&& R(W, \bar W, Z,\bar W)=g([W,\bar W]^{0,1},[Z,\bar W])\,,\\
&& R(W, \bar W, W,\bar W)=-|[W,\bar W]^{0,1}|_g^2\,,
\end{eqnarray}
for every $Z\in \f i^{1,0}$ and $W\in\left(\f i^{\perp}\right)^{1,0}$ with $|W|_g=1$. 
\end{lem}
\begin{proof}
Let $W\in (\f i^{\perp})^{1,0}$ with $|W|_g=1$. Then, since $\dim ((\f i^{\perp})^{1,0}) = 1$, equation \eqref{R} implies 
$$
R(W, \bar W, Z,\bar W)=g([W,\bar W]^{0,1},[Z,\bar W])-
g([W,\bar W],Z)g([W,\bar W],\bar W)
+g([[W,\bar W]^{1,0},\bar W],Z)\,,
$$
for every $Z\in \f i^{1,0}$. Since condition \eqref{bracket} implies that $[Z,\bar W]\in \f i^{1,0}$,  for every $Z\in \f i^{1,0}$, we have 
$$
R(W, \bar W, Z,\bar W)=g([W,\bar W]^{0,1},[Z,\bar W])-
g([W,\bar W],Z)g([W,\bar W],\bar W)
+g([[W,\bar W]^{1,0}_{\f i^{\perp}},\bar W],Z)\,.
$$
Furthermore, by writing 
$$
[W,\bar W]^{1,0}_{\f i^{\perp}}=\lambda W\,, \quad \lambda\in \C
$$
we have 
$$
R(W, \bar W, Z,\bar W)=g([W,\bar W]^{0,1},[Z,\bar W])-\lambda
g([W,\bar W],Z)
+\lambda g([W,\bar W],Z)=g([W,\bar W]^{0,1},[Z,\bar W])\,.
$$
Moreover, taking into account again  that $[Z,\bar W]\in \f i^{1,0}$,  for every $Z\in \f i^{1,0}$, we have 
$$
\begin{aligned}
R(W, \bar W, W,\bar W)=&-|[W,\bar W]^{0,1}|_g^2+
|g([W,\bar W],W)|^2
+g([[W,\bar W]^{1,0}_{\f i^{\perp}},\bar W],W)\\
=& -|[W,\bar W]^{0,1}|_g^2+
|\lambda|^2
+\lambda g([W,\bar W],W)\\
=& -|[W,\bar W]^{0,1}|_g^2+
|\lambda|^2-|\lambda|^2\\
=&-|[W,\bar W]^{0,1}|_g^2\,, 
\end{aligned}
$$
as required. 
\end{proof}
We have the following direct consequence of Lemma \ref{codim2} and Lemma \ref{main1}.

\begin{prop}\label{cod2} Let $M=\msf G / \Gamma$ be the compact quotient of a Lie group by a lattice. Let $(g,J)$ be a Hermitian structure on the Lie algebra $\g$ of $\msf G$. Assume that $\g$ admits a $J$-invariant ideal $\mathfrak{i}$ such  that ${\rm codim}_\R\f i=2$,   $[\ad_X, J]|_{\f i }=0$, for all $X\in \f g$  and 
\begin{equation}\label{fabio}
 g([W,\bar W]^{0,1},[Z,\bar W])=0\,,
\end{equation}
for every $Z\in\mathfrak{i}$, where $W$ is a generator of $(\f i^{\perp})^{1,0}$ with $|W|_g=1$. Then the pluriclosed flow on  $(M,J)$ has a long-time solution for every  
initial pluriclosed metric on $(M,J)$. 
\end{prop}

\medskip
We recall that a Lie algebra $\g$ is {\em almost-abelian} if it has a codimension-one abelian ideal $\f a $. If a complex structure $J$ on $\g$ is fixed, then $\f i:=\f a\cap J \f a$ is a $J$-invariant abelian ideal of $\g$ and $[\ad_X, J]|_{\f i }=0$,  for every $X\in \g$, see  \cite[Lemma 4.1]{AL19}. For every Hermitian metric on $(\g,J)$, we can choose an orthonormal basis $\{e_1, e_2, \ldots, e_{2n}\}$ of $(\f g,g)$ such that $e_{2n}\in \f a^{\perp}$, $e_1\in \f a \cap \f i^{\perp}$ and $Je_1=e_{2n}$. 
As a consequence, we can  write
$$
\ad_{e_{2n}}=\begin{pmatrix}
 a & 0 & 0 \\
 v& A & 0 \\
0&0&0
 \end{pmatrix}\,, \quad A:=\ad_{e_{2n}}|_{\f i}\in \f{gl}(\f i, J)\,, \quad a \in \R\,, \quad v \in \f i\,.
$$
and $W:=\frac{1}{\sqrt{2}}(e_1-\sqrt{-1}e_{2n})$ is a generator of $(\f i^{\perp})^{1,0}$ of unit norm.

With this set up, we are now in a position to prove the second point in Theorem \ref{main3}.

\begin{proof}[Proof of Theorem $\ref{main3}-2$]
In \cite{AL19}, it is proved that an almost-abelian Lie algebra $(\g,J)$ admits a pluriclosed inner product if and only if $A$ is a normal operator that commutes with $J$, and its eigenvalues have real part equal to 0 or $-\frac a2 $. Moreover in this case $(\f g,J)$ has a Hermitian inner product satisfying  \eqref{fabio}, see \cite[Proposition 4.2]{Par21} and point 2 of Theorem \ref{main3} follows from Proposition \ref{cod2}.
\end{proof}

We remark that Hermitian metrics on almost-abelian Lie algebras that satisfy \eqref{fabio} are \emph{locally conformally balanced}, i.e. ${\rm d} \theta=0$, where $\theta$ is the Lee form of $(g, J)$, defined by ${\rm d}\omega^{n-1}=\theta \wedge \omega^{n-1}$.

\medskip 

Finally, we consider the case where the Lie algebra $\g$ of $\msf G$ has a $J$-invariant codimension $2$ abelian  nilradical $\f i$. 

\begin{proof}[Proof of Theorem $\ref{main3}-3$]
In view of \cite{BF24}, in this hypothesis,  the existence of a pluriclosed inner product on $(\g,J)$ forces $[\ad_Y, J]|_{\f i}=0$, for any $Y\in\f g$. 
 In order to apply Proposition \ref{cod2},  we need to show the existence of a Hermitian metric $g$ on $(\g,J)$ such that $g([W,\bar W]^{0,1},[Z,\bar W])=0$, for every $Z\in \f i^{1,0}$, where $W$ is a  generator of $(\f i^\perp)^{1,0}$ with $|W|_g=1$. Writing $W=X-\sqrt{-1}JX$, with $X\in \f i^{\perp}$, the condition is equivalent to $g([X,JX],{\rm ad}_XY+J{\rm ad}_{JX}Y)=0$, for every $Y\in \f i$, i.e. 
$$
A_X^t\,[X,JX]=0\,,
$$
where $A_X:=({\rm ad}_X+J{\rm ad}_{JX})|_{\f i}$ and the  superscript $t$ denotes the transpose with respect to $g$. 

Let $h$ be a Hermitian metric on $\f g$ and fix a non-zero vector $X$ orthogonal to $\f i$ in order to split $\g$ as 
$$
\g=\langle X,JX\rangle \oplus \f i\,.
$$
Furthermore, we can decompose  $\f i={\rm Im}A_X\,\oplus\,({\rm Im}A_X)^{\perp_{h}}$. Let $A_XV$ be the component of $[X,JX]$ in ${\rm Im}A_X$ and consider $X':=X+JV$. Since $\ad_{JX}J=J\ad_{JX}$, we have 
$$
[X',JX']=[X,JX]-[X,V]-[JX,JV] =[X,JX]-A_XV\in ({\rm Im}A_X)^{\perp_{h}}=\ker A_X^t\,.
$$
Note that, since $\f i$ is abelian, we have $A_X=A_{X'}$. Finally, we define $g$ as the Hermitian metric  that coincides with   the restriction of $h$ on $\f i$, $|X'|_g=1$  and  makes  $\f i$ and $\langle X',JX'\rangle$ orthogonal. The claim follows. 
\end{proof}


\begin{rmk}\label{rmk_OT}
Another interesting  case where  Proposition \ref{main1}   can be applied is  when $\f i^\perp$ is a subalgebra of $\g$. In this case,  formula \eqref{R} implies that 
$$
R(U , \bar V , Z,\bar W)=0\,,
$$ for any 
$Z\in\mathfrak{i}^{1,0}$, 
$U,V,W\in (\mathfrak{i}^{\perp})^{1,0}$
and  conditions in Proposition \ref{main1} are equivalent to 
\begin{equation}\label{iperpSubalg}
\begin{cases}
& [\ad_X, J]|_{\f i }=0\,,\mbox{ for all }X\in \g;\\
& R(U, \bar U, W,\bar W)\leq 0\,,\mbox{ for all }U,W\in (\f i^{\perp})^{1,0}\,.
\end{cases}
\end{equation}
The condition of $\f i^\perp$ being a subalgebra is satisfied by  Oeljeklaus--Toma manifolds. In this case, $\g$ has a Hermitian metric with unitary $(1,0)$-basis $\{Z_1,\dots,Z_r,W_1,\dots,W_r\}$ satisfying the structure equations: 
$$
[W_k,\bar W_k]=\,-\frac{\sqrt{-1}}{2}(W_k+\bar W_k)\,,\quad 
[W_k,Z_i]=-\lambda_{ki}Z_i\,,\quad 
[ W_k,\bar  Z_i]=\bar \lambda_{ki} \bar Z_i\,,
$$
for $k,i=1,\dots, r$, where the $\lambda_{ki}$'s are suitable constants. Here $\f i^{1,0}=\langle Z_1, \ldots, Z_r\rangle$ and $(\f i^\perp)^{1,0}=\langle W_1, \ldots, W_r\rangle$
and the only non-zero components of the Chern curvature are $R(W_i,\bar W_i,W_i,\bar W_i)$ and $R(W_i,\bar W_i,Z_i,\bar Z_i)$. Moreover, $R(U, \bar U, W,\bar W)\leq 0$, for all $U,W \in (\f i^\perp)^{1,0}$. It follows that  conditions \eqref{iperpSubalg} are satisfied and, by applying Proposition \ref{main1},  we recover the result of Streets--Wang that solutions of the pluriclosed flow are immortal on Oeljeklaus--Toma manifolds, see \cite{SW25}. 
In the recent work \cite{FG26}, it is proved that any pluriclosed, 2-step solvable Lie algebra is the semidirect product of a negatively curved subalgebra and a Chern-Ricci flat ideal.  However, it is not clear in general if condition \eqref{bracket} holds on said ideal.

\end{rmk}

\section{Proof of Theorem \ref{surfaces}.}

The goal of this section is to prove Theorem \ref{surfaces} and Corollary \ref{mario}. Let $M$ be a compact complex manifold and $\pi \colon M \to X$ be a holomorphic principal $T_{\C}^r$-bundle over a compact  nonpositively curved K\"ahler manifold $(X, h)$. In the hypothesis of Theorem \ref{surfaces}, we show that 
the Chern curvature tensor  of  the metric
\begin{equation}\label{metricgfib}
g=\pi^*h+\sum_{i=1}^{2r}\theta_i\otimes\theta_i,
\end{equation}  
on $M$ satisfies the assumptions of Theorem \ref{main}. Note that in this case the vertical bundle $\mathcal V:=\ker {\rm d} \pi$ is pointwise spanned by the fundamental vector fields $X_i$ of the $T^r_{\mathbb C}$-action, while 
$$
\mathcal H:=\bigcap_{i=1}^{2r}\ker \theta_i\,.
$$   We will further assume that $J\theta_{2s-1}=\theta_{2s}$, for any $s=1,\ldots,r $. We should remark  here that, since the characteristic classes of $\pi$ are chosen to be of type $(1,1)$, the existence of at least a complex structure $J$ satisfying $J\theta_{2s-1}=\theta_{2s}$, for any $s=1,\ldots,r $ and making $\pi$ holomorphic is guaranteed by \cite[Lemma 1]{GGP08}. In what follows,  we will always denote with $Y, Z\in \Gamma(TX)$,  $V\in \Gamma(\mathcal V)$ and   $\tilde Z\in \Gamma(\mathcal H)$ the unique horizontal lift of $Z$. 
 
First of all, we show  that 
\begin{equation}\label{firstconditionfib}
R(V,\cdot,\cdot,\cdot)=0\,.
\end{equation}
 To do that,  we recall the following general formula for the Chern connection: for any $T, U, W\in \Gamma(TM)$, 
\begin{equation}\label{chern}
\begin{aligned}
2g(\nabla_TU, W)=&\, Tg(U, W)- (JT)g(JU, W)\\
&+g([T, U], W)+ g(J[JT, U], W)- g([T, W], U)- g(J[JT, W], U)\,.
\end{aligned}
\end{equation}
Using \eqref{chern}, and the fact that $\mathcal{L}_{X_i}g = \mathcal{L}_{X_i}J = 0$, we can infer that 
$$
g(\nabla_{X_i}\tilde Z, V)=0\,,
$$
for $i=1,\dots,2r$. Moreover 
$$
2g(\nabla_{X_{i}}\tilde Z, \tilde Y)=X_ig(\tilde Z, \tilde Y)-(JX_i)g(J \tilde Z,\tilde Y)=X_i\pi^*h(Z, Y)-(JX_i)\pi^*h(JZ, Y)=0\,.
$$  
It follows that 
\begin{equation}\label{Elia1}
\nabla_{X_{i}}\tilde Z=0\,.
\end{equation}
From this, we also find that 
$$
g(\nabla_{X_{i}}V, \tilde Y)=X_{i}g(V, \tilde Y)-g(V, \nabla_{X_{i}}\tilde Y)=0\,. 
$$ 
 Hence,  we have  that 
\begin{equation}\label{Elia2}
\nabla_{X_{i}}V\in \Gamma(\mathcal V)\,. 
\end{equation}
Since $[X_i, X_j]=0$ and  $g(X_i, X_j)=\delta_{ij}$, for every $i,j=1,\dots,2r$, formula \eqref{chern} directly implies 
$$
g(\nabla_{X_{i}}X_j, X_k)=0\,.
$$ 
 Using \eqref{Elia2} and the above formula, we obtain
\begin{equation}\label{Elia5}
\nabla_VX_i=0\,.
\end{equation}
Furthermore,  for any $s=1, \ldots, r$, we deduce that
\begin{equation}\label{Elia3}
\begin{aligned}
2g(\nabla_{\tilde Y}X_{2s-1}, \tilde Z)=&\, -g([\tilde Y, \tilde Z], X_{2s-1})-g(J[J\tilde Y, \tilde Z], X_{2s-1})\\
=&\, d\theta_{2s-1}(\tilde Y, \tilde Z)+dJ\theta_{2s-1}(J\tilde Y, \tilde Z)
=\pi^*( \omega_{2s-1}( Y,  Z)+ \omega_{2s}(J Y,  Z))
\end{aligned}
\end{equation}
and 
$$
d\pi(\nabla_{\tilde Y}X_{2s-1} )=\frac12 h^{-1}(\iota_{Y}\omega_{{2s-1}}+ \iota_{JY}\omega_{2s})\,.
$$ 
In addition, we can also conclude that 
\begin{equation}\label{Elia6}
g(\nabla_{\tilde Y}X_i, V)=0\,.
\end{equation}
Finally, we obtain

 \begin{equation}\label{Elia4}
 \begin{aligned}
 \nabla_{\tilde Y} \tilde Z=&\,  \widetilde{\nabla^{h}_{ Y} Z} +\sum_{i=1}^{2r}  g(\nabla_{\tilde Y} \tilde Z, X_i)X_i
 = \widetilde{\nabla^{h}_{ Y} Z}   -\sum_{i=1}^{2r}g(\nabla_{\tilde Y} X_{i}, \tilde Z)X_i\,.
 \end{aligned}
 \end{equation}
 Using $C^{\infty}$-linearity of $R$, in order to obtain \eqref{firstconditionfib}, it is enough to prove 
$$
R(X_i, U)W=0\,,
$$
when $W$ and $U$ are  lifts of a vector field on $X$ or  fundamental vectors of the $T^r_{\C}$-action. For such $W$ and $U$, we have 
$$
R(X_i, U)W=\nabla_{X_i}\nabla_UW\,. 
$$
In view of \eqref{Elia1} and \eqref{Elia2}, we may assume that $U=\tilde Y $  with $Y\in \Gamma(TX)$. Hence,  we need to verify that  
\begin{eqnarray}
&&  \nabla_{X_i}\nabla_{\tilde Y}\tilde Z=0\,,\label{sanji}\\ 
&&  \nabla_{X_i}\nabla_{\tilde Y} X_j=0\,,\label{doraimon}
\end{eqnarray}
for every $Z\in \Gamma(TX)$ and $i,j=1,\dots,2r$. 

Using  \eqref{Elia4} and \eqref{Elia6}, condition \eqref{sanji} is equivalent to   
$$
\nabla_{X_i}\nabla_{\tilde Y}\tilde Z=\nabla_{X_i}\left(\widetilde{\nabla^{h}_{ Y} Z}   -\sum_{j=1}^{2r}g(\nabla_{\tilde Y} X_j, \tilde Z)X_j\right)=
\nabla_{X_i}\widetilde{\nabla^{h}_{ Y} Z}   -\sum_{j=1}^{2r}g(\nabla_{\tilde Y} X_j, \tilde Z)\nabla_{X_i}X_j\,.
$$ 
Now, in view of \eqref{Elia1} and \eqref{Elia5}, we have that $\nabla_{X_i}\widetilde{\nabla^{h}_{ Y} Z} = 0$, and $\nabla_{X_i}X_j=0$, for any $i,j=1, \ldots, 2r$, forcing $\nabla_{X_i}\nabla_{\tilde Y}\tilde Z=0$. 

Finally, to prove \eqref{doraimon},  we use \eqref{Elia3} and \eqref{Elia1} to deduce  
$$
g(\nabla_{X_i}\nabla_{\tilde Y} X_j,\tilde Z)=X_ig(\nabla_{\tilde Y}X_j, \tilde Z)-g(\nabla_{\tilde Y}X_j, \nabla_{X_i}\tilde Z)=0\,,
$$
and \eqref{Elia6} and \eqref{Elia5} to obtain 
$$
g(\nabla_{X_i}\nabla_{\tilde Y} X_j,X_k)=0\,,
$$ concluding that $R(V, \cdot, \cdot, \cdot)=0$, for any $V\in \Gamma(\mathcal V)$.

We end the proof of Theorem \ref{surfaces} by showing that $R(\tilde Y, J\tilde Y, X_{2s-1}, \tilde Z)=0$, for every  $Y, Z\in \Gamma(TX)$ and $s=1, \ldots, r$. Using \eqref{Elia5} and \eqref{Elia3}, we have that
$$
g(\nabla_{[\tilde Y, J\tilde Y]}X_{2s-1}, \tilde Z)=\frac12\pi^*(h(h^{-1}(\iota_{[Y, JY]}\omega_{{2s-1}}+ \iota_{J[Y, JY]}\omega_{2s}), Z))\,,
$$ 
 while 
 $$
 \begin{aligned}
g(\nabla_{\tilde Y}\nabla_{J\tilde Y}X_{2s-1}, \tilde Z)=&\,\frac12 \pi^*(h(\nabla^h_Yh^{-1}(\iota_{JY}\omega_{2s-1}-\iota_Y\omega_{2s}), Z))\,,\\
g(\nabla_{J\tilde Y}\nabla_{\tilde Y}X_{2s-1}, \tilde Z)=&\,\frac12\pi^*(h(\nabla^h_{JY}h^{-1}(\iota_{Y}\omega_{2s-1}+\iota_{JY}\omega_{2s}), Z))\,.
\end{aligned}
 $$ 
 Now,  since $h$ is K\"ahler and the $\omega_i$'s are parallel, we infer that  
 $$
 \begin{aligned}
h&\, (\nabla^h_Yh^{-1}\iota_{JY}\omega_{2s-1}-\nabla^h_{JY}h^{-1}\iota_{Y}\omega_{2s-1}-h^{-1}(\iota_{[Y, JY]}\omega_{2s-1}), Z)\\
=&\, Y\omega_{2s-1}(JY, Z)-\omega_{2s-1}(JY, \nabla^h_YZ)-(JY)\omega_{2s-1}(Y, Z)+\omega_{2s-1}(Y, \nabla^h_{JY}Z)- \omega_{2s-1}([Y, JY], Z)\\
=&\, (\nabla^h_Y\omega_{2s-1})(JY, Z)-(\nabla^h_{JY}\omega_{2s-1})(Y, Z)=0\,.
\end{aligned}
$$ 
Analogously, 
$$
h(\nabla^h_Yh^{-1}(\iota_Y\omega_{2s})+\nabla^h_{JY}h^{-1}(\iota_{JY}\omega_{2s})+h^{-1}(\iota_{J[Y, JY]}\omega_{2s}), Z)=(\nabla^h_Y\omega_{2s})(Y, Z)+ (\nabla^h_{JY}\omega_{2s})(JY, Z)=0\,.
$$ 
Then, $R(\tilde Y, J \tilde Y, X_{2s-1}, \tilde Z)=0$. Hence, the metric $g$ defined in \eqref{metricgfib} satisfies the hypothesis of Theorem \ref{main}, concluding the proof of Theorem  \ref{surfaces}.

We conclude this section by showing Corollary \ref{mario}.

\begin{proof}[Proof of Corollary \ref{mario}]
For the first claim, let $\pi\colon M\to X$  be a holomorphic principal $T_{\C}^r$-bundle  as in the statement of Corollary \ref{mario}.  We claim there is a principal torus bundle connection on $M$ satisfying the hypotheses of Theorem \ref{surfaces}. Indeed, let  $\omega_{\textrm{KE}}$ be a non-positively curved K\"ahler--Einstein metric on $X$ scaled so that $H^2(X,\mathbb Z) = \mathbb Z[\omega_{\textrm{KE}}]$. Then, if $\theta_i \in \Lambda^1(M)$, $i=1, \ldots, 2r$, are the  principal torus bundle connections with $J\theta_{2s-1} = \theta_{2s}$, for $s=1,\ldots, r$,  their curvatures satisfy $\mathrm{d} \theta_i = \pi^* \alpha_i$, for $\alpha_i \in \Lambda^{1,1}(X)$ with $[\alpha_i] = k_i [\omega_{\textrm{KE}}], k_i \in \mathbb Z$, $i = 1,\ldots, 2r$. It follows that $\alpha_i = k_i \omega_{\textrm{KE}} + \mathrm{d}\mathrm{d}^c f_i$, with $f_i\in C^{\infty}( X ,  \mathbb R)$, $i =1,\ldots,  2r$. We may then modify the connection by defining $\theta_{2s-1}' := \theta_{2s-1} - \pi^*(\mathrm{d}^cf_{2s-1} + \mathrm{d}f_{2s})$ and $\theta_{2s}' := J \theta_{2s-1}'$, and easily verify that $\mathrm{d}\theta'_i = k_i \pi^* \omega_{\textrm{KE}}$, hence satisfying the hypothesis of Theorem \ref{surfaces}, from which the conclusion follows. 

Now, assume that $S$ is a minimal non-K\"ahler compact complex surface with $\kappa\ge 0 $.
Then $S$ can be finitely covered by a minimal non-K\"ahler properly elliptic surface $M$, see \cite[Lemma 1 and 2]{Briz94}. Thus, without loss of generality we may work on $M$, which is in particular a holomorphic  principal torus bundle over a Riemann surface $X$ with non-positive Euler characteristic. Applying the first part of Corollary \ref{mario} is sufficient to obtain the claim.
\end{proof}


 
\bibliographystyle{is-abbrv}
\bibliography{Tutto}
\end{document}